\documentclass[11pt, reqno]{amsart}
\usepackage[utf8]{inputenc}
\usepackage{amsmath,bm}
\usepackage{amsthm}
\usepackage{amsfonts}
\usepackage{MnSymbol}
\usepackage{blindtext,amsmath,amsthm,amsfonts, textcomp, mathrsfs, mathtools}
\usepackage[shortlabels]{enumitem}

\usepackage[top=30mm,bottom=30mm,left=30mm,right=30mm]{geometry}
\newtheorem{theorem}{Theorem}[section]
\newtheorem{conjecture}{Conjecture}

\newtheorem*{theorem*}{Theorem}
\newtheorem*{remark*}{Remark}
\newtheorem*{problem*}{Problem}
\newtheorem*{conjecture*}{Conjecture}
\newtheorem*{question*}{Question}
\newtheorem{lemma}[theorem]{Lemma}

\newtheorem{example}{Example}
\newcommand{\rom}[1]{\uppercase\expandafter{\romannumeral #1\relax}}

\begin{document}

\title[On the generalized Brauer-Siegel theorem]{On the generalized Brauer-Siegel theorem\\ for asymptotically exact families of number fields\\ with solvable Galois closure}

\author[Anup B. Dixit]{Anup B. Dixit}
\address{Department of Mathematics and Statistics\\ Queen's University\\ Jeffrey Hall, 48 University Ave,\\ Kingston,\\ Canada, ON\\ K7L 3N8}
\email{abd6@queensu.ca}

\date{}

\begin{abstract}
In 2002, M. A. Tsfasman and S. G. Vladut \cite{TV} formulated the generalized Brauer-Siegel conjecture for asymptotically exact families of number fields. In this article, we establish this conjecture for asymptotically good towers and asymptotically bad families of number fields with solvable normal closure.
\end{abstract}



\maketitle

\section{\bf Introduction}
\medskip

Let $K$ be an algebraic number field. Denote the class number of $K$ by $h_K$, the order of the ideal class group of $K$. It is an important theme in number theory to understand how $h_K$ varies on varying $K$. A prelude to this problem is Gauss's conjecture, settled independently by  Heegner \cite{Heeg}, Stark \cite{Stk2} and Baker \cite{Bak}, which states that there are exactly $9$ imaginary quadratic fields with class number $1$. Suppose $\mathcal{K} = \{K_i\}_{i\in\mathbb{N}}$ is a sequence of number fields. We call $\mathcal{K}$ to be a \textit{family} if $K_i\neq K_j$ for $i\neq j$. Gauss also predicted that in a family of imaginary quadratic fields, the class number $h_K$ must tend to infinity. This was shown by Heilbronn \cite{Heil} in 1934. This sparked the beginning of the study of asymptotic behaviour of the class number in a family of number fields. An immediate consequence of Heilbronn's result is that there are finitely many imaginary quadratic fields with a bounded class number.\\

However, the same phenomena is not expected to hold for any general family of number fields. For instance, it is still unknown whether there are infinitely many real quadratic fields with class number $1$, although it is widely believed to be true. One of the difficulties in bounding the class number is isolating it from the regulator of the number field. This was observed by Siegel \cite{Sie} in 1935. He showed that for a family of quadratic fields $\{K_i\}$, the class number times the regulator  $h_{K_i} R_{K_i}$ tends to infinity as $i\to \infty$. In other words, there are finitely many quadratic fields with bounded $h_K R_K$. In the case of real quadratic fields, the regulator is the $\log$ of the fundamental unit, where as in the case of imaginary quadratic fields, the regulator is $1$. Hence, this is a generalization of Heilbronn's result.\\

Furthermore, Siegel also established that if $\{K_i\}$ is a family of quadratic fields, then
\begin{equation*}
\lim_{i\to\infty} \frac{\log h_{K_i} R_{K_i}}{\log\sqrt{d_{K_i}}} = 1,
\end{equation*}
where $d_{K_i}$ denotes the absolute value of the discriminant $|disc(K/\mathbb{Q})|$. By Minkowski's theorem, we know that there are finitely many number fields with bounded discriminant. Hence, Siegel's result provides a rate at which $h_K R_K$ goes to infinity. Brauer \cite{BS} generalized this result to families of number fields, that are Galois over $\mathbb{Q}$. This is known as the classical Brauer-Siegel theorem. More precisely, he showed the following.
\begin{theorem}[Brauer]
Let $\{K_i\}$ be a family of number fields such that $K_i/\mathbb{Q}$ is Galois for all $i$. Denote by $n_{K_i}$ the degree $[K_i: \mathbb{Q}]$. If 
\begin{equation*}
    \lim_{i\to\infty} d_{K_i}^{1/n_{K_i}} = \infty,
\end{equation*} 
then 
\begin{equation}\label{brauer-siegel}
\lim_{i\to\infty} \frac{\log h_{K_i} R_{K_i}}{\log\sqrt{d_{K_i}}} = 1.
\end{equation}
Moreover, the condition $K_i/\mathbb{Q}$ being Galois can be dropped under the assumption of generalized Riemann hypothesis (GRH). 
\end{theorem}
The reason $h_K R_K$ appears in the above result is because of the class number formula. Recall the Dirichlet class number formula, which states that if $\rho_K$ denotes the residue of the Dedekind zeta-function $\zeta_K(s)$ at $s=1$, then
\begin{equation}\label{class-number formula}
    \rho_K = \frac{2^{r_1} (2\pi)^{r_2} h_KR_K}{\omega_K \sqrt{d_K}},
\end{equation}
where $r_1$ denotes the number of real embeddings and $r_2$ denotes the number of complex embeddings up to conjugation of $K$, and $\omega_K$ denotes the number of roots of unity in $K$. Using the class number formula, it is easy to see that equation \eqref{brauer-siegel} is equivalent to
\begin{equation}\label{brauer-siegel-2}
    \lim_{i\to\infty} \frac{\log \rho_{K_i}}{\log \sqrt{d_{K_i}}} =0.
\end{equation}

Now one would hope to show \eqref{brauer-siegel-2}, relying on the analytic behaviour of $\zeta_K(s)$ for certain families of number fields. The key is to be able to find a zero-free region of $\zeta_K(s)$ near $s=1$ for all $K$ in the family. In 1974, Stark \cite{Stk} exploited this idea to prove the Brauer-Siegel theorem for families of almost normal number fields, which do not contain any quadratic fields and also obtained effective growth of the class number $h_K$ for certain families of CM-fields.  A more extensive study of the Brauer-Siegel theorem, where the condition $d_{K_i}^{1/n_{K_i}} \to \infty$ can be dropped, was carried out by Tsfasman-Vl\u{a}du\c{t} \cite{TV} in 2002. They formulated \textit{the generalized Brauer-Siegel conjecture} for asymptotically exact families and proved it in the case of asymptotically good towers of almost normal number fields. The precise statement of their conjecture and the details of the background required will be discussed in Section 2. In 2005, Zykin \cite{Zyk} proved that the generalized Brauer-Siegel conjecture holds for asymptotically bad families of almost normal number fields.\\

In this paper, we prove the generalized Brauer-Siegel conjecture for asymptotically good towers as well as asymptotically bad families of number fields with solvable Galois closure. The main ingredient used is the result of V. K. Murty \cite{Km} on the zero-free region for $\zeta_K(s)$ near $s=1$, when $K$ has solvable Galois closure.

\section{\bf Notation}
\medskip

Let $K$ be a number field. We say $K$ is \textit{almost normal} if there exists a sequence of number fields $\{K_i\}$ such that 
\begin{equation*}
    \mathbb{Q} = K_0 \subseteq K_1\subseteq\cdots\subseteq K_n = K,
\end{equation*}
with all the $K_i/K_{i-1}$ normal, $1\leq i\leq n$.

Denote by $h_K$ the class number of $K$, $d_K$ the absolute value of the discriminant $|disc(K/\mathbb{Q})|$ and $R_K$ the regulator of $K$.
Define the genus of $K$ as
\begin{equation*}
g_K := \log \sqrt{d_K}.
\end{equation*}
Let $N_q(K)$ denote the number of non-archimedean places $v$ of $K$ such that $Norm(v)=q$.

For a number field $K/\mathbb{Q}$, the Dedekind zeta-function is defined as
\begin{equation*}
\zeta_K(s):= \prod_{\mathfrak{P}\subset \mathcal{O}_K} \left( 1- N\mathfrak{P}^{-s}\right)^{-1},
\end{equation*}
for $\Re(s)>1$, where $\mathfrak{P}$ runs over all non-zero prime ideals in the ring of integers of $K$. This can be re-written as
\begin{equation*}
    \zeta_K(s) = \prod_q \left(1- q^{-s}\right)^{-N_q(K)},
\end{equation*}
for $\Re(s)>1$, where $q$ runs over all prime powers. $\zeta_K(s)$ has an analytic continuation to the whole complex plane except for a simple pole at $s=1$ with residue $\rho_K$. Additionally, $\zeta_K(s)$ satisfies a functional equation akin to the Riemann zeta-function $\zeta(s)$, invariant under $s\mapsto 1-s$. Owing to the Euler product, $\zeta_K(s)\neq 0$ for $\Re(s)>1$. Using the functional equation, it can be shown that the only zeros of $\zeta_K(s)$ in $\Re(s)<0$ are the trivial zeros. The famous \textit{generalized Riemann hypothesis} (GRH) asserts that if $\zeta_K(s)=0$ and $0 \leq \Re(s)\leq 1$, then $\Re(s)=1/2$. In certain applications, the assumption of GRH can often be replaced by a weaker hypothesis of a zero-free region of $\zeta_K(s)$ near $s=1$. If there exists a real zero $\beta$ of $\zeta_K(s)$ satisfying
\begin{equation*}
    1-\frac{1}{4\, \log d_K} \leq \beta < 1,
\end{equation*}
then we say that $\beta$ is an \textit{exceptional} zero $\zeta_K(s)$. It is known that for any $\zeta_K(s)$, there is at most one such exceptional zero. In fact, the best known result in this context is due to Louboutin \cite{Loub3}, proving that there is at most one exceptional zero, $\beta$, satisfying
\begin{equation}\label{louboutin-zero}
    1>\beta \geq 1 - \frac{1}{c_0 \,g_K}, \hspace{2mm} \text{ where }\hspace{2mm}
    c_0 =  \frac{2\,(3+2\sqrt{2})}{5+\sqrt{5}} = 1.61\cdots.
\end{equation}

Suppose $\mathcal{K} = \{K_i\}_{i\in\mathbb{N}}$ is a sequence of number fields. We call $\mathcal{K}$ to be a family if $K_i\neq K_j$ for $i\neq j$. Moreover, we call  $\mathcal{K}$ to be a \textit{tower} if $K_i\subsetneq K_{i+1}$ for all $i$. We say that a family $\mathcal{K}$ is \textit{asymptotically exact} if the limits
\begin{equation*}
\phi_{\mathbb{R}} := \lim_{i\to \infty} \frac{r_1(K_i)}{g_{K_i}}, \hspace{5mm} \phi_{\mathbb{C}} := \lim_{i\to \infty} \frac{r_2(K_i)}{g_{K_i}},\hspace{5mm}
\phi_q := \lim_{i\to\infty} \frac{N_q(K_i)}{g_{K_i}}
\end{equation*}
exist for all prime powers $q$, where $r_1(K_i)$ and $r_2(K_i)$ are the number of real and complex embeddings of $K_i$ respectively.

We say that an asymptotically exact family $\mathcal{K} = \{K_i\}$ is \textit{asymptotically bad}, if $\phi_q = \phi_{\mathbb{R}} = \phi_{\mathbb{C}} =0$ for all prime powers $q$. This is equivalent to saying that the root discriminant $ d_{K_i}^{1/n_{K_i}}$ tends to infinity as $i\to \infty$. If an asymptotically exact family $\mathcal{K}$ is not asymptotically bad, we say that it is \textit{asymptotically good}.

The generalized Brauer-Siegel conjecture, as formulated by Tsfasman-Vl\u{a}du\c{t} \cite{TV} is as follows. 
\begin{conjecture}\label{BS}
For any asymptotically exact family $\mathcal{K}$, the limit
\begin{equation*}
BS(\mathcal{K}) := \lim_{i\to\infty} \frac{\log h_{K_i} R_{K_i}}{g_{K_i}}
\end{equation*}
exists and is equal to
\begin{equation}\label{BS1}
BS(\mathcal{K})=1 + \sum_q \phi_q \log\frac{q}{q-1} -\phi_{\mathbb{R}} \log 2 - \phi_{\mathbb{C}}\log 2\pi.
\end{equation}
Using the class number formula, the above statement is equivalent to the existence of the limit
\begin{equation*}
\rho(\mathcal{K}) := \lim_{i\to\infty} \frac{\log \rho_{K_i}}{g_{K_i}}
\end{equation*}
and 
\begin{equation}\label{BS2}
\rho(\mathcal{K}) =  \sum_q \phi_q \log\frac{q}{q-1}.
\end{equation}
\end{conjecture}

In the rest of the paper, we refer to the above conjecture as the GBS conjecture. Note that the GBS conjecture for asymptotically bad families is equivalent to the classical Brauer-Siegel conjecture. In \cite{TV}, Tsfasman-Vl\u{a}du\c{t} proved GBS for any asymptotically exact family $\mathcal{K}$ under the assumption of GRH. Unconditionally, they proved it for asymptotically good towers of almost normal number fields. Later in 2005, Zykin \cite{Zyk} showed GBS for asymptotically bad families of almost normal number fields.

\section{\bf Asymptotically exact families with solvable Galois closure}
\medskip

Let $K/\mathbb{Q}$ be a number field and $L\supseteq K \supseteq \mathbb{Q}$ be the normal closure of $K$ over $\mathbb{Q}$. We say that $K$ has solvable Galois closure if the Galois group $Gal(L/\mathbb{Q})$ is solvable. Recall that a group $G$ is said to be solvable if there exists subgroups $\{1\} =G_0\subseteq G_1\subseteq\cdots\subseteq G_m=G$ with $G_{i-1}$ a normal subroup of $G_i$ and $G_i/G_{i-1}$ abelian for $1\leq i\leq m$. We show the following.

\begin{theorem}\label{gbs-good-tower}
Let $\mathcal{K} = \{K_i\}$ be an asymptotically good tower of number fields, where each $K_i$ has solvable Galois closure over $\mathbb{Q}$. Then GBS holds for $\mathcal{K}$.
\end{theorem}

\begin{theorem}\label{gbs-bad-family}
Let $\mathcal{K} = \{K_i\}$ be an asymptotically bad family of number fields, where each $K_i$ has solvable Galois closure over $\mathbb{Q}$. Then GBS holds for $\mathcal{K}$.
\end{theorem}

We give a simple example to illustrate Theorem \ref{gbs-bad-family}.
\begin{example}
Let $K_n := \mathbb{Q}(2^{1/3}, 3^{1/3}, \cdots , p_n^{1/3})$, where $p_n$ is the n-th prime number. Then, $\{K_n\}$ forms an asymptotically bad family of number fields where each $K_n$ has solvable Galois closure. By Theorem \ref{gbs-bad-family}, GBS holds for $\{K_n\}$, i.e.,
\begin{equation*}
   \lim_{n\to \infty} \frac{\log \rho_{K_n}}{g_{K_n}} =0.
\end{equation*}
\end{example}

For an asymptotically bad family $\mathcal{K}$, GBS implies that $\rho(\mathcal{K})=0$. One of the natural questions is to determine the rate at which this limit converges to $0$. In this context, we show the following conditional result.

\begin{theorem}\label{GRH-effective}
Under the assumption of GRH, for an asymptotically bad family $\mathcal{K}= \{K_i\}$, we have 
\begin{equation}\label{grh-effective-equation}
\frac{\log \rho_{K_i}}{g_{K_i}} = O\left(\frac{\log g_{K_i}/n_{K_i}}{g_{K_i}/n_{K_i}}\right),
\end{equation}
where the implied constant only depends on $\mathcal{K}$.
\end{theorem}

Note that for an asymptotically bad family, $g_K/n_K$ tends to infinity and hence the right hand side in \eqref{grh-effective-equation} tends to $0$.

\section{\bf Preliminaries}
\medskip

In this section, we state and prove some results which will be useful in proofs of the main theorems. A crucial role in our proof is played by a result of V. K. Murty in \cite{Km}, a weaker version of which is as follows.
\begin{theorem}\label{Km}(Murty)
Suppose $K/\mathbb{Q}$ is an extension of degree $n$ whose Galois closure is solvable. Let
\begin{align*}
e(n) & := \max_{p^\alpha || n} \alpha, \\
\delta(n) & := (e(n) + 1)^2 \, \, 3^{1/3} \, \, 12^{e(n) -1}.
\end{align*}
There exists an absolute constant $c>0$, such that if $\zeta_K$ has a real zero in the region
\begin{equation}\label{region}
1 - \frac{c}{n^{e(n)} \delta(n) \log d_K} \leq \beta < 1,
\end{equation}
then there is a quadratic field $N\subseteq K$, such that $\zeta_N(\beta) = 0$.
\end{theorem}

We prove the following important lemma which connects the generalized Brauer-Siegel conjecture to zero-free regions for Dedekind zeta-functions. The proof of Lemma \ref{mainlemma} is inspired by \cite{TV} and uses their notation. For a number field $K$, write
\begin{equation}\label{zeta_k-around-1}
    \zeta_K(s) = \frac{\rho_K}{s-1}\,  F_K(s),
\end{equation}
where $F_K(s)$ is entire. Define
\begin{equation*}
Z_K(s) := \frac{d}{ds} \left( \frac{\log F_K(s)}{g_K} \right).
\end{equation*}

\begin{lemma}\label{mainlemma}
Let $K$ be a member of an asymptotically good family $\mathcal{K}$. Suppose the degree of $K$ is $n$ and $\zeta_K(s)$ has no zero in the region \eqref{region}. Then there exist absolute constants $C_1$, $C_2 $ and $C_3 >0$ dependent on $\mathcal{K}$, but independent of $K$, satisfying
\begin{equation*}
|Z_K(1+\theta)| \leq C_1\, g_K^{C_2 \log g_K},
\end{equation*}
for any $\theta\in (0,1)$ and any $g_K > C_3$.
\end{lemma}
\begin{proof}
Using Mellin transform of the Chebyshev step function, we have
\begin{equation}\label{L-O}
\frac{Z_K(s)}{s} = \frac{1}{g_K} \int_{1}^{\infty} (G_K(x) - x) \, x^{-s-1} \, dx - \frac{1}{s \, g_K},
\end{equation}
for $\Re(s) > 1$, where
\begin{equation*}
G_K(x) := \sum_{\substack{q,m \geq 1 \\ q^m \leq x}} N_q(K) \log q.
\end{equation*}
The unconditional Lagarias-Odlyzko \cite[Theorem 9.2]{Lag} estimate for $G_K(x)$ gives
\begin{equation*}
|G_K(x) - x| \leq C_4 \, x \exp \left( -C_5 \sqrt{\frac{\log x}{n}}\right) + \frac{x^\beta}{\beta}
\end{equation*}
for $\log x \geq C_6 \, n \, g_K^2$, where $C_4$, $C_5$, $C_6$ are positive absolute constants. Here, $\beta$ is the possible real exceptional zero of $\zeta_K(s)$. If such a zero does not exist, we set $\beta = 1/2$. By \eqref{louboutin-zero}, for $\zeta_K(s)$ with no zeroes in the region \eqref{region}, we have
\begin{equation*}
1 - \frac{c}{n^{e(n)} \delta(n) g_K} \geq \beta \geq 1 - \frac{1}{c_0\,g_K}.
\end{equation*}
For an asymptotically good family $\mathcal{K}$, note that we have $n_{K_i}/ g_{K_i}$ converges to a positive real number as $i\to\infty$. This is because if $n_{K_i}/ g_{K_i}\to 0$ as $i\to\infty$, then $\phi_{\mathbb{R}},\phi_{\mathbb{C}}$ and $\phi_{q}$ would be $0$ for all $q$, which contradicts that $\mathcal{K}$ is asymptotically good. Therefore, we can find positive constants $C_0$ and $C_{00}$ depending on $\mathcal{K}$ such that
\begin{equation*}
C_0 \, n_{K_i} \leq g_{K_i }\leq C_{00} \, n_{K_i},
\end{equation*}
for all $i$ large. Since $K$ is a member of $\mathcal{K}$, we have
\begin{equation*}
e(n) \leq \frac{\log n}{\log 2} \leq C_7 \, \log g_K
\end{equation*}
and
\begin{equation*}
\delta(n) \leq C_7 n^4 \leq C_8 g_K^4,
\end{equation*}
where $C_7, C_8$ are absolute positive constants.
Therefore, there exists $C_9, C_{10} > 0$, such that for 
\begin{equation*}
\log x \geq C_{10} \, g_K^{2 \, C_9 \, \log g_K },
\end{equation*}
we have
\begin{equation*}
\frac{x^{\beta}}{\beta} = o\left( x \exp \left(-C_5' \sqrt{\frac{\log x}{n}}\right)\right),
\end{equation*}
where $C_5' >0$. Let
\begin{equation*}
    g_K' := C_{10} \, g_K^{2 \, C_9 \, \log g_K }.
\end{equation*}
Setting $s = 1 + \theta$ in \eqref{L-O}, we have
\begin{equation}\label{integral}
    \left| \frac{Z(1+\theta)}{1+\theta} \right| \leq I_1 + I_2 + O(1),
\end{equation}
where
\begin{equation*}
   I_1 =  \frac{1}{g_K} \int_1^{g_K'} \left|G_K(x)-x\right| x^{-2-\theta} dx \hspace{5mm} \text{ and }\hspace{5mm}
   I_2 =  \frac{1}{g_K} \int_{g_K'}^{\infty} \left|G_K(x)-x\right| x^{-2-\theta} dx.
\end{equation*}
For $I_1$, we use the following bound on $G_K(x)$,
\begin{equation*}
G_K(x) = \sum_{\substack{q,m\geq 1\\ q^m\leq x}} N_q(K) \log q \leq n \sum_{\substack{q,m \geq 1\\ 
q^m \leq x}} \log q \ll g_K \, x \log x.
\end{equation*}
Therefore, for some constant $C_{11}>0$, we have
\begin{equation*}
|G_K(x) - x| \leq C_{11} \, g_K \, x \log x.
\end{equation*}
Thus the integral
\begin{equation*}
I_1 = \frac{1}{g_K} \int_1^{g_K'} |G_K(x)-x| x^{-2-\theta} dx \leq C_{11} \int_1^{g_K'} x^{-1-\theta} \log x \, dx \, \leq \, C_{11} \, g_K' \left(1 - \exp\left(-\theta g_K' \right)\right) \ll  \,{g_K'}^2.
\end{equation*}

We now show that the integral $I_2$ is bounded. By the Lagarias-Odlyzko estimate \eqref{L-O}, using the change of variables $x = y^{g_K \log y},$ we have
\begin{equation*}
I_2 \leq \frac{C_4}{g_K} \int_{g_K'}^{\infty} \exp \left(-C_5' \sqrt{\frac{\log x}{g_K}}\right) x^{-1-\theta} dx = 2 \, C_4 \int_{\exp \left( \sqrt{g_K'/g_K} \right)}^\infty \,\,
y^{-C_5' - 1 - \theta g_K \log y} \log y\, dy.
\end{equation*}
For large $g_K$ and any fixed $\epsilon>0$, we bound $\log y \leq y^{\epsilon}$ and get
\begin{equation}\label{bound}
 \leq 2C_4 \int_{\exp \left( \sqrt{g_K'/g_K} \right)}^\infty \,\, y^{-C_5' - 1+\epsilon - \theta g_K \log y} dy \leq 2C_4 \int_{\exp \left( \sqrt{g_K'/g_K} \right)}^\infty y^{-C_5' - 1+\epsilon - \theta \sqrt{g_K'}} dy.
\end{equation}
Evaluating the integral \eqref{bound}, we have
\begin{equation*}
\frac{2C_4}{-C_5' - 1+\epsilon - \theta \sqrt(g_K')} \exp \left( \sqrt{g_K'/g_K} \right)^{-C_5' +\epsilon - \theta \sqrt{g_K'}} \leq C_{12}.
\end{equation*}
for some absolute constant $C_{12}$. Thus, we have the lemma.
\end{proof}

In order to find a zero-free region for all the Dedekind zeta-functions attached to number fields in an asymptotically good family, we prove a crucial lemma below.
\begin{lemma}\label{quad-fields}
Let $\mathcal{K} = \{K_i\}$ be an asymptotically good family of number fields. Set
\begin{equation*}
Q(\mathcal{K}) := \{k; \, [k:\mathbb{Q}]=2 \text{ and } k \subseteq K_i \text{ for some } i \}.
\end{equation*}
Then, $Q(\mathcal{K})$ is a finite set.
\end{lemma}
\begin{proof}
Since $\mathcal{K}$ is asymptotically good, we have
\begin{equation*}
\lim_{i\to\infty} \frac{n_{K_i}}{g_{K_i}} >0.
\end{equation*}
Thus, there exists a fixed $\epsilon > 0$, such that $n_{K_i}/g_{K_i} > \epsilon$ for all $i$. Since $k\subseteq K_i$, we have
\begin{equation*}
\frac{2}{g_k} = \frac{n_k}{g_k} \geq \frac{n_{K_i}}{g_{K_i}} > \epsilon.
\end{equation*}
Thus, $g_k \leq 2/ \epsilon$. Hence, $Q(\mathcal{K})$ is finite.
\end{proof}

A vital role in our proof is also played by the following result of Stark (see \cite[Lemma 4]{Stk}).
\begin{lemma}\label{Stklem}(Stark)
There exists an effectively computable constant $c'>0$ such that for any number field $K$, we have
\begin{equation*}
\rho_K > c'(1-\beta),
\end{equation*}
where $\beta$ is the possible exceptional zero of $\zeta_K(s)$. If such a zero does not exist, then we set $\beta = 1/2$.
\end{lemma}

Moreover, a theorem of Louboutin \cite{Lou1} regarding an upper bound for residues of Dedekind zeta-functions is of significance and hence, is stated below.
\begin{theorem}\label{upbound}(Louboutin)
Let $K$ be a number field. If $\zeta_K(\beta) = 0$ for $1/2 \leq \beta <1$, then
\begin{equation}
\rho_K \leq (1-\beta) \left( \frac{e g_K}{2n_K} \right)^{n_K}.
\end{equation}
\end{theorem}
The following conditional bound is utilized in the proof of Theorem \ref{GRH-effective}.
\begin{lemma}\label{grheffect}
Under the assumption of GRH, for any number field $K$, we have
\begin{equation*}
\frac{|Z_K(1+\theta)|}{1+\theta} \ll 1
\end{equation*}
for $\theta\in (0,1)$.
\end{lemma}
\begin{proof}
For any number field $K$, from \eqref{L-O}, we have  
\begin{equation}\label{odlyz}
\frac{Z_K(s)}{s} = \frac{1}{g_K} \int_{1}^{\infty} (G_K(x) - x) x^{-s-1} dx - \frac{1}{sg_K},
\end{equation}
for $\Re(s) > 1$, where
\begin{equation*}
G_K(x) := \sum_{\substack{q,m \geq 1 \\ q^m \leq x}} N_q(K) \log q.
\end{equation*}
The Lagarias-Odlyzko estimate assuming GRH (see \cite[Theorem 9.1]{Lag}) for $G_K(x)$ gives
\begin{equation*}
|G_K(x)-x|\leq c \, g_K \, x^{1/2} \, (\log x)^2.
\end{equation*}
Using this estimate in \eqref{odlyz}, we get the lemma.
\end{proof}

\section{\bf Proof of main theorems}
\medskip

\subsection{\textit{Proof of Theorem \ref{gbs-good-tower}}}
Taking $\log$ on both sides of \eqref{zeta_k-around-1} and dividing by $g_K$, we get for $s = 1+\theta_K$
\begin{equation}\label{main}
\frac{\log \zeta_K (1+\theta_K)}{g_K} = \frac{\log \rho_K}{g_K} + \frac{\log F_K(1+ \theta_K)}{g_K} - \frac{\log \theta_K}{g_K}.
\end{equation}
Here, and in the rest of the paper, $\log$ is chosen to be the principal branch. In \cite{TV}, it is shown that for any asymptotically exact family of number fields,
\begin{equation*}
\limsup_{i\to\infty} \frac{\log \rho_{K_i}}{g_{K_i}} \leq \sum_q \phi_q \log \frac{q}{q-1}.
\end{equation*}
Therefore, in order to prove Theorem \ref{gbs-good-tower}, it suffices to show that for an asymptotically good tower $\mathcal{K} = \{K_i\}$ of number fields with solvable Galois closure,
\begin{equation*}
\liminf_{i\to\infty} \frac{\log \rho_{K_i}}{g_{K_i}} \geq \sum_q \phi_q \log \frac{q}{q-1}.
\end{equation*}
Hence by \eqref{main}, for a suitable choice of $\theta_{K_i} \to 0$, we are reduced to showing that,
\begin{equation}\label{zeta}
\liminf_{i\to \infty} \frac{\log \zeta_{K_i}(1+\theta_{K_i})}{g_{K_i}} \geq \sum_q \phi_q \log \frac{q}{q-1},
\end{equation}
\begin{equation}\label{F(s)}
\limsup_{i\to\infty} \frac{\log F_{K_i}(1+ \theta_{K_i})}{g_{K_i}} \leq 0,
\end{equation}
and
\begin{equation}\label{theta}
\lim_{i\to\infty} \frac{\log \theta_{K_i}}{g_{K_i}} = 0.
\end{equation}

We first prove \eqref{F(s)} and make our choice of $\theta_{K_i}$'s. From Lemma \ref{quad-fields}, it is clear that if we consider an asymptotically good tower of number fields $\{K_i\}$ where each $K_i$ has solvable Galois closure, there are at most finitely many of them with $\zeta_{K_i}$ having zeroes in the region \eqref{region}. So hereafter we will assume that our tower $\mathcal{K}$ does not have any number field $K$ such that $\zeta_K$ has a zero in the region \eqref{region}.

Choosing $\theta_K$ as
\begin{equation*}
\theta_K := g_K^{-(C_2+1) \log g_K},
\end{equation*}
using Lemma \ref{mainlemma} and the fact that $F_K(1)=1$, we get
\begin{equation*}
\frac{\log F_K(1+\theta_K)}{g_K} = \int_{0}^{\theta_K} Z_K(1+\theta) \,d\theta \ll g_K^{-\log g_K}.
\end{equation*}
Therefore, \eqref{F(s)} holds. Furthermore, we have
\begin{equation*}
\log \theta_{K_i} \ll (\log g_{K_i})^2.
\end{equation*}
Hence, we also get \eqref{theta}. For \eqref{zeta}, note that
\begin{equation*}
\frac{\zeta_{K_i}(1+\theta)}{g_{K_i}}  = \sum_q \frac{N_q(K_i)}{g_{K_i}} \log \frac{1}{1-q^{-1-\theta}} 
 = \sum_p  \frac{N_p(K_i)}{g_{K_i}} \log \frac{1}{1-p^{-1-\theta}} + \sum_{\substack{q=p^{k}, \\ k>1}} \frac{N_q(K_i)}{g_{K_i}} \log \frac{1}{1-q^{-1-\theta}}.
\end{equation*}
In a tower, we know that $\phi_p \leq \frac{N_p(K_i)}{g_{K_i}}$. Therefore,
\begin{equation*}
\sum_p  \frac{N_p(K_i)}{g_{K_i}} \log \frac{1}{1-p^{-1-\theta}} \geq \sum_p \phi_p  \log \frac{1}{1-p^{-1-\theta}},
\end{equation*}
for any $\theta>0$. We also have
\begin{equation*}
\sum_{\substack{q=p^{k},\\ k>1}} \frac{N_q(K_i)}{g_{K_i}} \log \frac{1}{1-q^{-1-\theta}} \to \sum_{\substack{q=p^{k}, \\ k>1}} \phi_q \log \frac{1}{1-q^{-1-\theta}} 
\end{equation*}
uniformly for $\theta > -\delta$, for some $\delta>0$. Hence, we get
\begin{equation*}
\liminf_{i\to \infty} \zeta_{K_i}(1+\theta_{K_i}) \geq \sum_q \phi_q \log \frac{q}{q-1}.
\end{equation*}

\subsection{\textit{Proof of Theorem \ref{gbs-bad-family}}}
Let $\mathcal{K} = \{K_i\}$ be an asymptotically bad family of number fields. If $K_i$'s do not have zeroes in the region \eqref{region}, then Lemma \ref{Stklem} gives
\begin{equation*}
\rho_{K_i} > c' (1- \beta) > c'  \frac{c}{n^{e(n)} \delta(n) \log d_K}.
\end{equation*}
Since $g_{K_i} \to \infty$ and $n_{K_i}/g_{K_i} \to 0$, we have the desired result.

Suppose some $K_i$ has zero in the region \eqref{region}, say $\beta_i$. Then, by Theorem \ref{Km}, there is a quadratic sub-field $k_i$ of $K_i$, which also has a zero at $\beta_i$. Now, using Theorem \ref{upbound} stated in Section 4, we get 
\begin{equation*}
\rho_{K_i} = \left(\frac{\rho_{K_i}}{\rho_{k_i}}\right) \rho_{k_i} \geq \left( \frac{c(1-\beta_i)}{(1-\beta_i) \left(  \frac{e g_{K_i}}{4}\right)^2} \right) \rho_{k_i}.
\end{equation*}

Taking $\log$ and dividing by $g_{K_i}$, we have
\begin{equation*}
\frac{\log \rho_{K_i}}{g_{K_i}} \geq \frac{\log{c  \left(  \frac{4}{e g_{K_i}}\right)^2}}{g_{K_i}} + \frac{\log \rho_{K_i}}{g_{K_i}}.
\end{equation*} 

Now using the classical Brauer-Siegel theorem for quadratic fields, we are done.

\subsection{\textit{Proof of Theorem \ref{GRH-effective}}}
We start with the equation \eqref{main}, as in the proof of Theorem \ref{gbs-good-tower}. Recall the definition
\begin{equation*}
Z_K(s) := \frac{d}{ds} \left(\frac{\log F_K(s)}{g_K}\right).
\end{equation*}
Using Lemma \ref{grheffect}, we have
\begin{equation}\label{term2}
\left|\frac{\log F_K(1+ \theta_K)}{g_K}\right| = \left|\int_0^{\theta_K} \left(\frac{\log F_K(1+ \theta)}{g_K}\right)' d\theta\right| \ll \theta_K.
\end{equation}


Since,
\begin{equation*}
    0< - \frac{\zeta_K'}{\zeta_K}(\sigma) \leq -n_K \frac{\zeta'}{\zeta}(\sigma) < \frac{n_K}{\sigma - 1} 
\end{equation*}
for $\sigma > 1$, using \cite[Lemma (a)]{Loub2}, we have
\begin{equation}\label{term1}
    0 < \frac{\log \zeta_K (1+\theta_K)}{g_K}  = \frac{1}{g_K} \int_{1+\theta_K}^{\infty} - \frac{\zeta_K'}{\zeta_K}(\sigma)\, d\sigma \, \leq \, \frac{n_K}{g_K} \int_{1+\theta_K}^{\infty} \frac{d\sigma}{\sigma-1}
    = \frac{n_K}{g_K} \log\left(\frac{1}{\theta_K}\right)
\end{equation}
Choosing
\begin{equation*}
\theta_K = \frac{\log g_K/n_K}{g_K/n_K},
\end{equation*}
and using \eqref{term1} and \eqref{term2}, we have
\begin{equation*}
\frac{\log \rho_{K}}{g_{K}} = O\left(\frac{\log g_K/n_K}{g_K/n_K}\right).
\end{equation*}

\section{\bf Bounds on regulators}
\medskip

As an application of the generalized Brauer-Siegel theorem, we follow the methods in \cite{TV} to produce some lower bounds on the regulators of number fields with solvable Galois closure in asymptotically good towers.  Proposition 7.1 of \cite{TV} states that for an asymptotically exact family $\mathcal{K} = \{K_i\}$ of number fields,
\begin{equation*}
    \limsup_{i\to\infty} \frac{\log {h_{K_i}}}{g_{K_i}} \leq 1- \left(\log 2\sqrt{\pi} + \frac{\gamma + 1}{2} \right) \phi_{\mathbb{R}} - \left(\log 4\pi + \gamma \right)\phi_{\mathbb{C}} + \sum_q \phi_q \log \frac{q}{q-1}.
\end{equation*}
Comparing this with the Theorem \ref{gbs-good-tower}, we have
\begin{theorem}
For an asymptotically good tower $\mathcal{K} = \{K_i\}$ of number fields with solvable Galois closure,
\begin{equation*}
    \liminf_{i\to\infty} \frac{\log R_{K_i}}{g_{K_i}} \geq \left(\log \sqrt{\pi e} + \frac{\gamma}{2}\right) \phi_{\mathbb{R}} + (\log 2 +\gamma) \phi_{\mathbb{C}}.
\end{equation*}
\end{theorem}

\begin{center}
\section{CONCLUDING REMARKS}
\end{center}

The approach used in the proof of Theorem \ref{gbs-good-tower} and Theorem \ref{gbs-bad-family} does not give any information on the rate at which  $\log \rho_{K_i}/ g_{K_i}$ tends to its limit $\sum_q \phi_q \log q/(q-1)$. In \cite{Stk}, Stark showed that for an asymptotically bad family almost normal fields $K$ not containing any quadratic subfield,
\begin{equation*}
    \frac{\log \rho_K}{g_K} = O \left(\frac{\log g_K}{g_K} \right),
\end{equation*}
where the implied constant is independent of $K$. Hence, we get some information on the rate at which $\log \rho_{K_i} /g_{K_i}$ converges to $0$ in such a family. It is interesting to investigate a similar question in more generality. Unfortunately, no such result is known for asymptotically good families. For an asymptotically bad family of number fields with solvable Galois closure, one may use \cite[Theorem 3.1]{Km} to give a partial result for a large sub-class of these number fields. However, this question still remains open in general.\\

\section{\bf Acknowledgement}
\medskip

I would like to thank Prof. V. K. Murty for his insightful comments on an earlier version of this paper. I am also grateful to both the anonymous referees whose suggestions helped in improving the exposition significantly.

\end{document}